\documentclass[12pt, reqno]{amsart}
\usepackage{amsmath, amsthm, amscd, amsfonts, amssymb, graphicx, color}
\usepackage[bookmarksnumbered, colorlinks, plainpages]{hyperref}
\hypersetup{colorlinks=true,linkcolor=red, anchorcolor=green, citecolor=cyan, urlcolor=red, filecolor=magenta, pdftoolbar=true}

\newtheorem{theorem}{Theorem}[section]
\newtheorem{lemma}[theorem]{Lemma}

\theoremstyle{definition}
\newtheorem{definition}[theorem]{Definition}

\theoremstyle{remark}
\newtheorem{remark}[theorem]{Remark}
\numberwithin{equation}{section}

\begin{document}

\setcounter{page}{1}
\title[Multilinear fractional integral operators]{Multilinear fractional integral operators on
non-homogeneous metric measure spaces}

\author[Huajun Gong, Rulong Xie \MakeLowercase{and} Chen Xu]{Huajun Gong,$^{1,2}$ Rulong Xie$^{3,4}$$^{*}$ \MakeLowercase{and} Chen Xu$^{1,2}$}

\address{$^{1}$College of Information Engineering, Shenzhen University, Shenzhen 518060, PR China;
\newline
$^{2}$College of Mathematics and Statistics, Shenzhen University, Shenzhen 518060, PR China.}
\email{\textcolor[rgb]{0.00,0.00,0.84}{huajun84@hotmail.com; xuchen\_szu@szu.edu.cn}}

\address{$^{3}$School of Mathematical Sciences, University of
Science and Technology of China, Hefei 230026, PR China;
\newline
$^{4}$Department of Mathematics, Chaohu University, Hefei 238000, PR China.}
\email{\textcolor[rgb]{0.00,0.00,0.84}{xierl@mail.ustc.edu.cn}}

%\dedicatory{This paper is dedicated to Professor ABCD}

%\let\thefootnote\relax\footnote{Copyright 2016 by the Tusi Mathematical Research Group.}



\date{Feb. 18th, 2016.
\newline \indent $^{*}$Corresponding author}

\begin{abstract}
Let $(X,d,\mu)$ be a non-homogeneous metric measure space satisfying both the geometrically doubling and the upper doubling measure
conditions. In this paper, the boundedness of multilinear fractional integral operator in this
setting is proved. Via a sharp maximal operator, the
boundedness of commutators generated by multilinear fractional integral operator with
$RBMO(\mu)$ function on non-homogeneous metric measure spaces in
Lebesgue spaces is obtained.
\end{abstract} \maketitle

\section{Introduction and preliminaries}
It is well known that a space of
homogeneous type is the space, which satisfies the assumption of the doubling
measure condition, i.e. there exists a constant $C>0$ such that
$\mu(B(x,2r))\leq C\mu(B(x,r))$ for all $x\in \rm{supp} \mu$ and
$r>0$. However non-doubling measure is a nonnegative
measure $\mu$ only satisfies the polynomial growth condition,
i.e., for all $x \in X$ and $r>0$, there exists a constant $C>0$ and
$k\in(0, n]$ such that,
\begin{equation}\label{equ1}
\mu(B(x, r)) \leq C_{0}r^{k},
\end{equation}
where $B(x, r)= \{y \in X: |y-x| < r\}$. This brings
rapid development in harmonic analysis (see \cite{CS,GM,
HMY1,HMY2,NTV,T1,T2,XS}). As an important application, it is to
solve the long-standing open Painlev$\acute{e}$'s problem (see
\cite{T1}).

   In \cite{H}, Hyt\"{o}nen pointed out that the doubling measure is not the special case of the non-doubling measures.
To overcome this difficulty, a kind of metric measure space $(X,d,\mu)$, that satisfies the
geometrically doubling and the upper doubling measure conditions (see
Definition 1.1 and 1.2) is introduced by Hyt\"{o}nen in \cite{H}, which is called the non-homogeneous metric measure space. The
highlight of this kind of space is that it includes both space of
homogeneous type and metric spaces with polynomial growth measures
as special cases. From then on, a lot of results paralleled to homogeneous
spaces and non-doubling measure spaces are obtained
(see\cite{BD,FLYY,FYY1,FYY2,H,HM,HLYY,HYY,LY1,LY2,LY3,LYY} and the references therein). For example, Hyt\"{o}nen et al. \cite{HYY} and Bui and Duong\cite{BD} independently
introduced the atomic Hardy space $\mathcal H^{1}(\mu)$ and obtained that the
dual space of $\mathcal H^{1}(\mu)$ is $RBMO(\mu)$. Bui and Duong\cite{BD} also
proved that Calder\'{o}n-Zygmund operator and commutators of
Calder\'{o}n-Zygmund operator with RBMO function are bounded in
$L^{p}(\mu)$ for $1<p<\infty$. Later, Lin and Yang \cite{LY1} introduced
the space $RBLO(\mu)$ and proved the maximal Calder\'{o}n-Zygmund
operator is bounded from $L^{\infty}(\mu)$ into $RBLO(\mu)$.
Recently, some equivalent characterizations was established by Liu  et al. \cite{LYY} for the boundedness of
Carder¡äon-Zygmund operators on $L^{p}(\mu)$ for $1<p<\infty$. Fu et al. \cite{FYY1,FYY2} established the boundedness of multilinear commutators of Calder¡äon-Zygmund
operators and commutators of generalized fractional integrals with $RBMO(\mu)$. Fu et al. \cite{FLYY}
partially established the theory of the Hardy space $\mathcal H^p$ with $p\in(0,1]$ on $(X,d,\mu)$.  The readers can refer to the survey
\cite{YYF} and the monograph \cite{YYH} for more developments on harmonic analysis in non-homogeneous metric measure spaces.

   At the other hand, the theory on multilinear integral operators has been studied by some researchers. Coifman
and Meyers\cite{CM} firstly established the theory of bilinear
Calder\'{o}n-Zygmund operators. Later, Gorafakos and Torres \cite{GT1,GT2}
established the boundedness of multilinear singular integral on the
product Lebesgue spaces and Hardy spaces. Xu\cite{X1,X2} established the properties of
multilinear singular integrals and commutators on non-doubling
measure spaces $(R^{n},\mu)$. The bounedeness of multilinear fractional integral and commutators on non-doubling
measure spaces $(R^{n},\mu)$ was proved by Lian and Wu in \cite{LW}. In  non-homogeneous metric measure spaces,
Hu et al. \cite{HMY3} established  the weighted norm inequalities for multilinear Calder\'{o}n-Zygmund operators.
The boundedness of commutators of multilinear singular integrals on Lebesgue spaces was obtained by Xie et al. in \cite{XGZH}.

 In this paper, multilinear fractional integral operator and commutators
generated by multilinear fractional integral with $RBMO(\mu)$
function on non-homog\\
eneous metric spaces are introduced.
And it is proved that multilinear farctional integral operators and
commutators are bounded in Lebesgue spaces on non-homogeneous metric spaces, provided that factional integral is bounded from
 $L^{r}$ into $L^{s}$ for all $r\in(1,\,1/\beta)$ and
$1/s=1/r-\beta$ with $0<\beta<1$. The results in
this paper include the corresponding results on both the
homogeneous spaces and $(R^{n},\mu)$ with non-doubling measure
spaces.

    We first recall some notations and definitions.

\begin{definition}\label{def1.1}
\cite{H} A metric spaces $(X,d)$
is called geometrically doubling if there exists some $N_{0}\in
\mathbf{N}$ such that, for any ball $B(x,r)\subset X$, there exists
a finite ball covering $\{B(x_i,r/2)\}_i$ of $B(x,r)$ such that the
cardinality of this covering is at most $N_0$.
\end{definition}

\begin{definition}\label{def1.2}
\cite{H} A metric measure space
$(X,d,\mu)$ is said to be upper doubling if $\mu$ is a Borel measure
on $X$ and there exists a function $\lambda : X
\times(0,+\infty) \rightarrow (0,+\infty)$ and a constant
$C_{\lambda} >0$ such that for each $x\in X, r \longmapsto (x,r)$ is
non-decreasing, and for all $x\in X, r >0$,
 \begin{equation}
 \mu(B(x, r))\leq \lambda(x, r)\leq C_{\lambda}\lambda(x, r/2).
\end{equation}
\end{definition}

\begin{remark}
(i)\quad A space of homogeneous type is a
special case of upper doubling spaces, where one can take $\lambda(x, r)\equiv \mu(B(x,r))$. On the other
hand, a metric space $(X,d,\mu)$ satisfying the polynomial growth
condition (\ref{equ1})(in particular, $(X,d,\mu)\equiv(R^n, |\cdot|, \mu)$
with $\mu$ satisfying (\ref{equ1}) for some $k\in (0, n])$) is also an
upper doubling measure space if we take $\lambda(x, r)\equiv
Cr^{k}$.

(ii) \quad Let$(X,d,\mu)$ be an upper doubling space and $\lambda$
be a function on $X \times(0,+\infty)$ as in Definition
1.2. In \cite{H}, it was showed that there exists another
function $\tilde{\lambda}$ such that for all $x, y \in X$ with $d(x,
y)\leq r$,
\begin{equation}\label{equ3}
\tilde{\lambda}(x, r)\leq \tilde{C}\tilde{\lambda}(y, r).
\end{equation}
 Thus, in this paper, we always suppose that $\lambda$ satisfies
(\ref{equ3}) and $\lambda(x, ar) \geq a^{m}\lambda(x, r)$ for all $ x \in X $ and $ a,r
> 0$.

(iii) As shown in \cite{TL}, the upper doubling condition
is equivalent to the weak growth condition:
there exists a function $\lambda: X \times(0,\infty)\to(0,\infty)$,
with $r\to\lambda(x,r)$ non-decreasing, a positive constant $C_\lambda$ depending on $\lambda$
and $\epsilon$ such that
\begin{itemize}
\item[(a)] for all $r\in(0,\infty)$, $t\in[0,r]$, $x,\,y\in X$ and $d(x,y)\in[0,r]$,
$$|\lambda(y,r+t)-\lambda(x,r)|\le C_\lambda\biggl[\frac{d(x,y)+t}r\biggr]^{\epsilon}
\lambda(x,r);$$

\item[(b)] for all $x\in X$ and $r\in(0,\infty)$,
$$\mu(B(x,r))\le\lambda(x,r).$$
\end{itemize}
\end{remark}

\begin{definition}\quad
 Let $\alpha,\beta \in (1,+\infty)$ and a ball $B\subset X$ is called $(\alpha,
\beta)$-doubling if $\mu(\alpha B)\leq \beta \mu (B)$.

As in Lemma 2.3 of \cite{BD}, there exist plenty of doubling balls
with small radii and with large radii. Throughout this paper,
unless $\alpha$ and $\beta$ are specified otherwise, by an
$(\alpha,\beta)$ doubling ball we mean a $(6,\beta_{0})$-doubling
with a fixed number $\beta_0 >\max\{C_{\lambda}^{3log_{2}6},
6^{n}\}$, where $n=log_{2}N_{0}$ be viewed as a geometric dimension
of the spaces.
\end{definition}

\begin{definition}\cite{FYY2}\quad Let $0\leq\gamma<1$. For any two balls
$B\subset Q$, set $N_{B,Q}$ be the smallest integer
satisfying $6^{N_{B,Q}}r_{B}\geq r_Q$, then one defines
\begin{equation}
K^{(\gamma)}_{B,Q} = 1+\sum_{k=1}^{N_{B,Q}}
\biggl[\frac{\mu(6^{k}B)}{\lambda(x_B,6^{k}r_{B})}\biggr]^{(1-\gamma)}.
\end{equation}
For $\gamma=0$, we simply write $K^{(0)}_{B,Q}=K_{B,Q}$.

The multilinear fractional integral on
nonhomogeneous metric measure spaces is defined as follows.
\end{definition}

\begin{definition}\quad
 Let $\alpha\in (0,m)$. A kernel
 \begin{equation*}
 K(\cdot,\cdots,\cdot)\in L_{loc}^{1}\left((X)^{m+1}\backslash\{(x,y_{1}\cdots,y_{j},\cdots,y_{m}):x=y_j, 1\leq j\leq m\}\right)
 \end{equation*}
 is called an $m$-linear fractional integral kernel if it satisfies:

(i)\begin{equation}\label{equ6}
 \quad |K(x,y_{1},\cdots,y_{j},\cdots, y_{m})|\leq
\frac{C}{\biggl[\sum_{j=1}^{m}\lambda(x,d(x,y_{j}))\biggr]^{m-\alpha}}
 \end{equation}
for all $(x,y_{1}\cdots,y_{j},\cdots,y_{m})\in (X)^{m+1}$ with
$x\neq y_{j}$ for some $j$.

 (ii) There exists $0<\delta\leq 1$ such that
\begin{eqnarray}\label{equ7}
&&|K(x,y_{1},\cdots,y_{j},\cdots,y_{m})-K(x',y_{1},\cdots,y_{j},\cdots,y_{m})|\nonumber\\
&\leq&\frac{Cd(x,x')^{\delta}}{\biggl[\sum\limits_{j=1}^{m}d(x,y_{j})\biggr]^{\delta}\biggl[\sum\limits_{j=1}^{m}\lambda(x,d(x,y_{j}))\biggr]^{m-\alpha}},
\end{eqnarray}
proved that $Cd(x,x')\leq \max\limits_{1\leq j\leq m}d(x,y_{j})$ and
for each $j$,
 \begin{eqnarray}\label{equ8}
&&|K(x,y_{1},\cdots,y_{j},\cdots,y_{m})-K(x,y_{1},\cdots,y'_{j},\cdots,y_{m})|\nonumber\\
 &\leq&\frac{Cd(y_{j},y'_{j})^{\delta}}{\biggl[\sum\limits_{j=1}^{m}d(x,y_{j})\biggr]^{\delta}
 \biggl[\sum\limits_{j=1}^{m}\lambda(x,d(x,y_{j}))\biggr]^{m-\alpha}},
\end{eqnarray}
proved that $Cd(y_{j},y'_{j})\leq \max\limits_{1\leq j\leq
m}d(x,y_{j})$.

A multilinear operators $I_{\alpha,m}$ is called the multilinear
fractional integral operator with the above
kernel $K$ satisfying (\ref{equ6}), (\ref{equ7}) and (\ref{equ8}) if, for $f_{1},\cdots
f_{m}$ are $L^{\infty}$ functions with compact support and $x\notin
\bigcap_{j=1}^{m}supp f_{j}$,
 \begin{eqnarray}\label{equ9}
 &&I_{\alpha,m}(f_{1},\cdots f_{m})(x)\nonumber\\
 &&=\int_{X^{m}}K(x,y_{1},\cdots y_{m})f_{1}(y_{1})\cdots
 f_{m}(y_{m})d\mu(y_{1})\cdots d\mu(y_{m}).
\end{eqnarray}
For $m=1$, we simply write $I_{\alpha,1}$ by $I_{\alpha}$, which is the generalized fractional integral operator introduced by \cite{FYY2}.
\end{definition}

\begin{remark}\quad Because $\max\limits_{1\leq j\leq
m}d(x,y_{j})\leq \sum\limits_{j=1}^{m}d(x,y_{j})\leq
m\max\limits_{1\leq j\leq m}d(x,y_{j})$, (ii) in Definition 1.5 is
equivalent to (ii') in the following statement.

 (ii') There exists $0<\delta\leq 1$ such that

 \begin{eqnarray*}
 &&|K(x,y_{1},\cdots,y_{j},\cdots,y_{m})-K(x',y_{1},\cdots,y_{j},\cdots,y_{m})|\\
 &\leq & \frac{Cd(x,x')^{\delta}}{\biggl[\max\limits_{1\leq j\leq
m}d(x,y_{j})\biggr]^{\delta}\biggl[\sum\limits_{j=1}^{m}\lambda(x,d(x,y_{j}))\biggr]^{m-\alpha}},
\end{eqnarray*}
proved that $Cd(x,x')\leq \max\limits_{1\leq j\leq m}d(x,y_{j})$ and
for each $j$,
\begin{eqnarray*}
&&|K(x,y_{1},\cdots,y_{j},\cdots,y_{m})-K(x,y_{1},\cdots,y'_{j},\cdots,y_{m})|\\
 &\leq &\frac{Cd(y_{j},y'_{j})^{\delta}}{\biggl[\max\limits_{1\leq j\leq
m}d(x,y_{j})\biggr]^{\delta}
 \biggl[\sum\limits_{j=1}^{m}\lambda(x,d(x,y_{j}))\biggr]^{m-\alpha}},
\end{eqnarray*}
proved that $Cd(y_{j},y'_{j})\leq \max\limits_{1\leq j\leq
m}d(x,y_{j})$.
\end{remark}

\begin{definition}\cite{BD}\quad Let $\rho>1$ be some
constant. A function $b\in L_{loc}^{1}(\mu)$ is said to belong
to $RBMO(\mu)$ if there exists a constant $C
>0$ such that for any ball $B$
\begin{equation}\label{equ11}
\frac{1}{\mu(\rho B)}\int_{B}|b(x)-m_{\tilde{B}}b|d\mu(x)\leq C,
\end{equation}
and for any two doubling ball $B\subset Q$,
 \begin{equation}\label{equ12}
  |m_{B}(b)-m_{Q}(b)|\leq CK_{B,Q}.
\end{equation}
$\widetilde{B}$ is the smallest $(\alpha,\beta)$-doubling ball of
the form $6^{k}B$ with $k\in {\mathbf{N}}\bigcup\{0\}$, and
$m_{\widetilde{B}}(b)$ is the mean value of $b$ on $\widetilde{B}$,
namely,
$$m_{\widetilde{B}}(b)=\frac{1}{\mu(\widetilde{B})}\int_{\widetilde{B}}b(x)d\mu(x).$$
The minimal constant $C$ appearing in (\ref{equ11}) and (\ref{equ12}) is defined to be
the $RBMO(\mu)$ norm of $f$ and denoted by $||b||_{\ast}$.
\end{definition}

For $1\leq i \leq m$, denote by $C_{i}^{m}$ the family of all
finite subsets
$\sigma=\{\sigma(1),\sigma(2),\cdot\cdot\cdot,\sigma(i)\}$ of
$\{1,2,\cdot\cdot\cdot,m\}$ with $i$ different elements. For any
$\sigma\in C_{i}^{m}$, the complementary sequences $\sigma'$ is
given by $\sigma'=\{1,2,\cdot\cdot\cdot,m\}\backslash\sigma$.
Moreover, for $b_{i}\in RBMO(\mu),i=1,\cdots,m$, let
$\vec{b}=(b_{1},b_{2},\cdot\cdot\cdot,b_{m})$ be a finite family of
locally integrable functions. For all $1\leq i\leq m$ and
$\sigma=\{\sigma(1),\cdot\cdot\cdot,\sigma(i)\}\in C_{i}^{m}$, we
set $\vec{b}_{\sigma}=(b_{\sigma(1)},\cdot\cdot\cdot,b_{\sigma(i)})$
and the product $b_{\sigma}(x)=b_{\sigma(1)}(x)\cdot\cdot\cdot
b_{\sigma(i)}(x)$. Also, we denote $\vec{f}=(f_{1},\cdots, f_{m})$,
$\vec{f}_{\sigma}=(f_{\sigma(1)},\cdots, f_{\sigma(i)})$ and
$\vec{b}_{\sigma'}\vec{f}_{\sigma'}=(b_{\sigma'(i+1)}f_{\sigma'(i+1)},\cdots,
b_{\sigma'(m)}f_{\sigma'(m)})$.

\begin{definition}\quad A kind of commutators generated
by multilinear fractional integral operators $I_{\alpha,m}$ with $b_{i}\in
RBMO(\mu), i=1,\cdots, m$ is defined as follows:
\begin{equation*}
[\vec{b},I_{\alpha,m}](\vec{f})(x)=\sum_{i=0}^{m}\sum_{\sigma \in
C_{i}^{m}}(-1)^{m-i}b_{\sigma}(x)I_{\alpha,m}(\vec{f}_{\sigma},\vec{b}_{\sigma'}\vec{f}_{\sigma'})(x).
\end{equation*}
\end{definition}

In particular, when $m=2$, it is easy to see that
\begin{equation}\label{equ15}
 \begin{split}
[b_{1},b_{2},I_{\alpha,2}](f_{1},f_{2})(x)=&b_{1}(x)b_{2}(x)I_{\alpha,2}(f_{1},f_{2})(x)-b_{1}(x)I_{\alpha,2}(f_{1},b_{2}f_{2})(x)\\
&-b_{2}(x)I_{\alpha,2}(b_{1}f_{1},f_{2})(x)+I_{\alpha,2}(b_{1}f_{1},b_{2}f_{2})(x).
\end{split}
\end{equation}
$[b_{1},I_{\alpha,2}]$ and $[b_{2},I_{\alpha,2}]$ are defined as follows respectively.
\begin{equation*}[b_{1},I_{\alpha,2}](f_{1},f_{2})(x)=b_{1}(x)I_{\alpha,2}(f_{1},f_{2})(x)-I_{\alpha,2}(b_{1}f_{1},f_{2})(x),\end{equation*}
\begin{equation*}[b_{2},I_{\alpha,2}](f_{1},f_{2})(x)=b_{2}(x)I_{\alpha,2}(f_{1},f_{2})(x)-I_{\alpha,2}(f_{1},b_{2}f_{2})(x).\end{equation*}

Without loss of generality, we only
consider the case of $m=2$. Now let us state the main
results.

\begin{theorem}\label{thm-main1.1}
Let $0<\alpha<2$, $1<p_{1}$, $p_{2}<+\infty$, $f_{1}\in
L^{p_{1}}(\mu)$ and $f_{2}\in L^{p_{2}}(\mu)$. If $I_\beta$ is bounded from $L^{r}$ into $L^{s}$ for all $r\in(1,\,1/\beta)$ and
$1/s=1/r-\beta$ with $0<\beta<1$, then
there exists a constant $C>0$ such that
\begin{equation*}
||I_{\alpha,2}(f_{1},f_{2})||_{L^{q}(\mu)}\leq C
||f_{1}||_{L^{p_{1}}(\mu)}||f_{2}||_{L^{p_{2}}(\mu)},
\end{equation*}
where $\dfrac{1}{q}=\dfrac{1}{p_{1}}+\dfrac{1}{p_{2}}-\alpha$.
\end{theorem}

\begin{theorem}\label{thm-main1.2}
Set $\mu$ is a Radon measure with $||\mu||=\infty$,
$0<\alpha<2$, $1<p_{1}$, $p_{2}<+\infty$,
$f_{1}\in L^{p_{1}}(\mu)$, $f_{2}\in L^{p_{2}}(\mu)$, $b_{1},b_{2}\in
RBMO(\mu)$ and if $I_\beta$ is bounded from $L^{r}$ into $L^{s}$ for any $r\in(1,\,1/\beta)$,
$1/s=1/r-\beta$ with $0<\beta<1$,  then there
exists a constant $C>0$ such that
\begin{equation*}
||[b_{1},b_{2},I_{\alpha,2}](f_{1},f_{2})||_{L^{q}(\mu)}\leq C
||f_{1}||_{L^{p_{1}}(\mu)}||f_{2}||_{L^{p_{2}}(\mu)},
\end{equation*}
where $\dfrac{1}{q}=\dfrac{1}{p_{1}}+\dfrac{1}{p_{2}}-\alpha$.
\end{theorem}

\begin{remark}\quad  For $||\mu||<\infty$, by Lemma 2.1 in Section 2, Theorem \ref{thm-main1.2} also holds if one assumes that $\int_{X}G(f_1,f_2)(x)d\mu(x)=0$ with the operator $G$ be replaced by $I_{\alpha,2}$, $[b_1,I_{\alpha,2}]$,
$[b_2,I_{\alpha,2}]$ and $[b_1,b_2,I_{\alpha,2}]$.
\end{remark}

This paper is organized as follows. Theorem \ref{thm-main1.1} and Theorem \ref{thm-main1.2} are proved in Section 2.  In Section 3, some applications are stated. Throughout this paper, $C$ always denotes a
positive constant independent of the main parameters involved, but
it may be different from line to line.

\section{Proof of Main Results}
\begin{proof}[Proof of Theorem \ref{thm-main1.1}] Let $\alpha=\alpha_{1}+\alpha_{2}$, $0<\alpha_i<1/p_i<1$ for $i=1,2$. It is easy to check that
\begin{equation*}
\prod_{j=1}^{2}[\lambda(x,d(x,y_{j}))]^{1-\alpha_i}\leq
\biggl[\sum_{j=1}^{2}\lambda(x,d(x,y_{j}))\biggr]^{2-\alpha},
\end{equation*}
thus
\begin{equation*}
|I_{\alpha,2}(f_{1},f_{2})(x)|\leq \prod_{j=1}^{2}I_{\alpha_i}(|f_{i}|)(x).
\end{equation*}
Let $1/q_i=1/p_i-\alpha_i$ and $1/q_{1}+1/q_2=1/q$, $1<q_{i}<\infty$. It follows from the H\"{o}lder's inequality and the boundedness of $I_{\alpha_i},i=1,2$ that
\begin{eqnarray*}
&&||I_{\alpha,2}(f_{1},f_{2})||_{L^{q}(\mu)}\\
&\leq& \bigg\|\prod_{j=1}^{2}I_{\alpha_i}(|f_{i}|)\bigg\|_{L^{q}(\mu)}\\
&\leq &||I_{\alpha_1}(|f_{1}|)||_{L^{q_1}(\mu)}||I_{\alpha_2}(|f_{2}|)(x)||_{L^{q_2}(\mu)}\\
&\leq& ||f_{1}||_{L^{p_1}(\mu)}||f_{2}||_{L^{p_2}(\mu)}.
\end{eqnarray*}
The proof of Theorem \ref{thm-main1.1} is completed.
\end{proof}

To prove Theorem \ref{thm-main1.2}, first we give some notations and lemmas.

Let $f \in L_{loc}^{1}(\mu)$ and $0<\beta<1$, the sharp maximal operator is
 \begin{equation*}
M^{\sharp,(\beta)}f(x)=\sup _{B\ni
x}\frac{1}{\mu(6B)}\int_{B}|f(y)-m_{\widetilde{B}}(f)|d\mu(y)
+\sup_{(B,Q)\in\Delta_{x}}\frac{|m_{B}(f)-m_{Q}(f)|}{K^{(\beta)}_{B,Q}},
\end{equation*}
where $\Delta_{x}:=\{(B,Q):x\in B\subset Q\ \text{and} \ B, \ Q\ \text{are
doubling balls}\}$.\\

 The non centered doubling maximal operator is
$$Nf(x)=\sup_{B\ni x,\ B \
doubling}\frac{1}{\mu(B)}\int_{B}|f(y)|d\mu(y).$$
 By the Lebesgue
differential theorem, for any $f\in L^{1}_{loc}(\mu)$, we have
\begin{equation*}
|f(x)|\leq Nf(x)
\end{equation*} for $\mu-a.e. \ x \in X$.

Set $\rho>1$, $p\in (1,\infty)$ and $r\in (1,p)$, the non-centered
maximal operator $M^{(\alpha)}_{r,(\rho)}f$ is defined by
\begin{equation*}
M^{(\alpha)}_{r,(\rho)}f(x)=\sup_{B\ni x}\biggl\{\frac{1}{[\mu(\rho
B)]^{1-\alpha r}}\int_{B}|f(y)|^{r}d\mu(y)\biggr\}^{1/r}.
\end{equation*}
When $r=1$, we simply write $ M^{(0)}_{1,(\rho)}f(x)$ as $M_{(\rho)}f$. If
$\rho\geq 5$, then the operator $M_{(\rho)}f$ is bounded on
$L^{p}(\mu)$ for any $p>1$, and $M^{(\alpha)}_{r,(\rho)}$ is bounded from $\L^{p}(\mu)$
to $\L^{q}(\mu)$ for $p\in (r,1/\alpha)$ and $1/q=1/p-\alpha$(see \cite{FYY2}).

\begin{lemma}\cite{FYY2}%lemma2.1
For $||\mu||<\infty$, if $f \in L^{1}_{loc}(\mu)$, $\int_{X}f(x)d\mu(x)=0$, $1<p<\infty$, $0<\delta<1$, and
$\inf(1,Nf)\in L^{p}(\mu)$, for $0<\beta<1$, then there exists a constant
$C>0$ such that
\begin{equation*}
||N(f)||_{L^{p}(\mu)}\leq
C||M^{\sharp,(\beta)}(f)||_{L^{p}(\mu)}.
\end{equation*}
\end{lemma}

\begin{lemma}\cite{BD,T2} $1\leq p<\infty$ and $1<\rho <\infty$, then $b\in RBMO(\mu)$ if
and only if for any ball $B\in X$,
\begin{equation*}
\biggl\{\frac{1}{\mu(\rho
B)}\int_{B}|b_{B}-m_{\tilde{B}}(b)|^{p}d\mu(X)\biggr\}^{1/p}\leq
C||b||_{\ast},
\end{equation*}
and for any two doubling ball $B\subset Q$,
\begin{equation}\label{equ44}
|m_{B}(b)-m_{Q}(b)|\leq C K_{B,Q}||b||_{\ast}.
\end{equation}
\end{lemma}

\begin{lemma}\cite{HMY1}
\begin{equation*}
|m_{\widetilde{6^{j}\frac{6}{5}B}}(b)-m_{\tilde{B}}(b)|\leq
Cj||b||_{\ast}.
\end{equation*}
\end{lemma}

\begin{lemma}
Suppose  $0<\alpha<2$, $1<p_{1},\ p_{2},\ q<\infty$,\ $1<r<q$ and $b_{1},\ b_{2}\in RBMO(\mu)$. If $I_\beta$ is bounded from $L^{r}$ into $L^{s}$ for all $r\in(1,\,1/\beta)$ and
$1/s=1/r-\beta$ with $0<\beta<1$, then there exists a constant $C>0$ such that
for any $x\in X$, $f_{1}\in L^{p_{1}}(\mu)$ and $f_{2}\in
L^{p_{2}}(\mu)$,
\begin{eqnarray}\label{equ30}
&&M^{\sharp,(\alpha/2)}[b_{1},b_{2},I_{\alpha,2}](f_{1},f_{2})(x)\\
&&\leq C\big\{||b_{1}||_{\ast}||b_{2}||_{\ast}M_{r,(6)}(I_{\alpha,2}(f_{1},f_{2}))(x)
+||b_{1}||_{\ast}M_{r,(6)}([b_{2},I_{\alpha,2}](f_{1},f_{2}))(x)\nonumber\\
&&+||b_{2}||_{\ast}M_{r,(6)}([b_{1},I_{\alpha,2}](f_{1},f_{2}))(x)
+||b_{1}||_{\ast}||b_{2}||_{\ast}M^{(\alpha/2)}_{p_{1},(5)}f_{1}(x)M_{p_{2},(5)}f_{2}(x)\big\},\nonumber
\end{eqnarray}

\begin{eqnarray}\label{equ31}
&&M^{\sharp,(\alpha/2)}[b_{1},I_{\alpha,2}](f_{1},f_{2})(x)]\\
&&\leq C\big\{||b_{1}||_{\ast}M_{r,(6)}(I_{\alpha,2}(f_{1},f_{2}))(x)
+||b_{1}||_{\ast}M^{(\alpha/2)}_{p_{1},(5)}f_{1}(x)M^{(\alpha/2)}_{p_{2},(5)}f_{2}(x)\big\},\nonumber
\end{eqnarray}
and
\begin{eqnarray}\label{equ32}
&&M^{\sharp,(\alpha/2)}[b_{2},I_{\alpha,2}](f_{1},f_{2})(x)\\
&&\leq
C\big\{||b_{2}||_{\ast}M_{r,(6)}(I_{\alpha,2}(f_{1},f_{2}))(x)
+||b_{2}||_{\ast}M^{(\alpha/2)}_{p_{1},(5)}f_{1}(x)M^{(\alpha/2)}_{p_{2},(5)}f_{2}(x)\big\}.\nonumber
\end{eqnarray}
\end{lemma}

\begin{proof}
As $L^{\infty}(\mu)$ with compact support is dense in
$L^{p}(\mu)$ for $1<p<\infty$, we only consider
$f_{1},\ f_{2}\in L^{\infty}(\mu)$ with compact support. Also, by
Corollary 3.11 in [4], without loss of generality, we assume
$b_{1},\ b_{2}\in L^{\infty}(\mu)$.

 As Theorem 9.1 in \cite{T2}, in order to obtain (\ref{equ30}), it suffices to show that
\begin{eqnarray}\label{equ33}
&&\frac{1}{\mu(6B)}\int_{B}||[b_{1},b_{2},
I_{\alpha,2}](f_{1},f_{2})(z)-h_{B}||d\mu(z)\\
&&\leq C\big\{||b_{1}||_{\ast}||b_{2}||_{\ast}M_{r,(6)}(I_{\alpha,2}(f_{1},f_{2}))(x)
+||b_{1}||_{\ast}M_{r,(6)}([b_{2},I_{\alpha,2}](f_{1},f_{2}))(x)\nonumber\\
&&+||b_{2}||_{\ast}M_{r,(6)}([b_{1},I_{\alpha,2}](f_{1},f_{2}))(x)+C||b_{1}||_{\ast}||b_{2}||_{\ast}M^{(\alpha/2)}_{p_{1},(5)}f_{1}(x)M^{(\alpha/2)}_{p_{2},(5)}f_{2}(x)\big\},\nonumber
\end{eqnarray}
 holds for
any $x\in B$, and
\begin{eqnarray}\label{equ34}
&&|h_{B}-h_{Q}|\\
&& \leq CK_{B,Q}^{2}K_{B,Q}^{(\alpha/2)}\biggr[||b_{1}||_{\ast}||b_{2}||_{\ast}M_{r,(6)}(I_{\alpha,2}(f_{1},f_{2}))(x)\nonumber\\
&&+||b_{1}||_{\ast}||b_{2}||_{\ast}M^{(\alpha/2)}_{p_{1},(5)}f_{1}(x)M^{(\alpha/2)}_{p_{2},(5)}f_{2}(x)\nonumber\\
&&+||b_{1}||_{\ast}M_{r,(6)}([b_{2},I_{\alpha,2}](f_{1},f_{2}))(x)
+||b_{2}||_{\ast}M_{r,(6)}([b_{1},I_{\alpha,2}](f_{1},f_{2}))(x)\biggr].\nonumber
\end{eqnarray}
for any ball $B\subset Q$ with $x\in B$, where $Q$ is a doubling ball. For any ball $B$, denote
$$h_{B}:= m_{B}(I_{\alpha,2}((b_{1}-
m_{\tilde{B}}(b_{1}))f_{1}\chi_{X\backslash\frac{6}{5}B},(b_{2}-m_{\tilde{B}}(b_{2}))f_{2}\chi_{X\backslash\frac{6}{5}B})),$$
and
$$h_{Q}:= m_{Q}(I_{\alpha,2}((b_{1}-
m_{Q}(b_{1}))f_{1}\chi_{X\backslash\frac{6}{5}Q},(b_{2}-m_{Q}(b_{2}))f_{2}\chi_{X\backslash\frac{6}{5}Q})).$$

Since
$$[b_{1},b_{2},I_{\alpha,2}]=I_{\alpha,2}((b_{1}-b_{1}(z))f_{1},(b_{2}-b_{2}(z))f_{2}),$$
and
\begin{eqnarray}
&&I_{\alpha,2}((b_{1}-m_{\tilde{B}}(b_{1}))f_{1},(b_{2}-m_{\tilde{B}}(b_{2}))f_{2})\\
&&=I_{\alpha,2}((b_{1}-b_{1}(z)+b_{1}(z)-m_{\tilde{B}}(b_{1}))f_{1},(b_{2}-b_{2}(z)+b_{2}(z)-m_{\tilde{B}}(b_{2}))f_{2})\nonumber\\
&&=(b_{1}(z)-m_{\tilde{B}}(b_{1}))(b_{2}(z)-m_{\tilde{B}}(b_{2}))I_{\alpha,2}(f_{1},f_{2})\nonumber\\
&&-(b_{1}(z)-m_{\tilde{B}}(b_{1}))I_{\alpha,2}(f_{1},(b_{2}-b_{2}(z))f_{2})\nonumber\\
&&-(b_{2}(z)-m_{\tilde{B}}(b_{2}))I_{\alpha,2}((b_{1}-b_{1}(z))f_{1},f_{2})\nonumber\\
&&+I_{\alpha,2}((b_{1}-b_{1}(z))f_{1},(b_{2}-b_{2}(z))f_{2}),\nonumber
\end{eqnarray}
it follows that
\begin{eqnarray}\label{equ36}
&&\biggl(\frac{1}{\mu(6B)}\int_{B}|[b_{1},b_{2},I_{\alpha,2}](f_{1},f_{2})(z)-h_{B}|d\mu(z)\biggr)\nonumber\\
&&\leq C\biggl(\frac{1}{\mu(6B)}\int_{B}|(b_{1}(z)-m_{\tilde{B}}(b_{1}))(b_{2}(z)-m_{\tilde{B}}(b_{2}))I_{\alpha,2}(f_{1},f_{2})(z)|d\mu(z)\biggr)\nonumber\\
&&+C\biggl(\frac{1}{\mu(6B)}\int_{B}|(b_{1}(z)-m_{\tilde{B}}(b_{1}))I_{\alpha,2}(f_{1},(b_{2}-b_{2}(z))f_{2})(z)|d\mu(z)\biggr)\nonumber\\
&&+C\biggl(\frac{1}{\mu(6B)}\int_{B}|(b_{2}(z)-m_{\tilde{B}}(b_{2}))I_{\alpha,2}((b_{1}-b_{1}(z))f_{1},f_{2})(z)|d\mu(z)\biggr)\nonumber\\
&&+C\biggl(\frac{1}{\mu(6B)}\int_{B}|I_{\alpha,2}((b_{1}-m_{\tilde{B}}(b_{1}))f_{1},(b_{2}-m_{\tilde{B}}(b_{2}))f_{2})(z)-h_{B}|d\mu(z)\biggr)\nonumber\\
&&=:E_{1}+E_{2}+E_{3}+E_{4}.
\end{eqnarray}
For $E_{1}$, let $1<r_{1},\ r_{2}$ such that
$\dfrac{1}{r}+\dfrac{1}{r_{1}}+\dfrac{1}{r_{2}}=1$.
It follows from H\"older's inequality that

\begin{eqnarray*}
&&E_{1}\\
&&\leq
C\biggl(\frac{1}{\mu(6B)}\int_{B}|b_{1}(z)-m_{\tilde{B}}b_{1}|^{r_{1}}d\mu(z)\biggr)^{1/r_{1}}\\
&&\quad\times\biggl(\frac{1}{\mu(6B)}\int_{B}|b_{2}(z)-m_{\tilde{B}}b_{2}|^{r_{2}}d\mu(z)\biggr)^{1/r_{2}}\\
&&\quad\times\biggl(\frac{1}{\mu(6B)}\int_{B}|I_{\alpha,2}(f_{1},f_{2})|^{r}d\mu(z)\biggr)^{1/r}\\
&&\leq C||b_{1}||_{\ast}||b_{2}||_{\ast}M_{r,(6)}(I_{\alpha,2}(f_{1},f_{2}))(x).
\end{eqnarray*}

For $E_{2}$, let $1<s$ such that
$\dfrac{1}{s}+\dfrac{1}{r}=1$, by H\"{o}lder's
inequality, one deduces

\begin{eqnarray*}
&&E_{2}\\
&&\leq C\biggl(\frac{1}{\mu(6B)}\int_{B}|b_{1}(z)-m_{\tilde{B}}b_{1}|^{s}d\mu(z)\biggr)^{1/s}\\
&&\  \ \times\biggl(\frac{1}{\mu(6B)}\int_{B}|[b_{2},I_{\alpha,2}](f_{1},f_{2})|^{r}d\mu(z)\biggr)^{1/r}\\
&&\leq C||b_{1}||_{\ast}M_{r,(6)}([b_{2},I_{\alpha,2}](f_{1},f_{2}))(x).
\end{eqnarray*}

For $E_3$, in a similar way we can obtain
\begin{equation*}
E_{3}\leq C||b_{2}||_{\ast}M_{r,(6)}([b_{1},I_{\alpha,2}](f_{1},f_{2}))(x).
\end{equation*}

For $E_{4}$, let
$f_{k}^{1}=f_{k}\chi_{\frac{6}{5}B}$ and $f_{k}^{2}=f_{k}-f_{k}^{1}$
for $k=1,2$. Then
\begin{eqnarray*}
&&E_{4}\\
&&\leq
C\biggl(\frac{1}{\mu(6B)}\int_{B}|I_{\alpha,2}((b_{1}-m_{\tilde{B}}b_{1})f_{1}^{1}(z),(b_{2}-m_{\tilde{B}}b_{2})f_{2}^{1})(z)|d\mu(z)\biggr)\\
&&\  +
C\biggl(\frac{1}{\mu(6B)}\int_{B}|I_{\alpha,2}((b_{1}-m_{\tilde{B}}b_{1})f_{1}^{1}(z),(b_{2}-m_{\tilde{B}}b_{2})f_{2}^{2})(z)|d\mu(z)\biggr)\\
&&\  +
C\biggl(\frac{1}{\mu(6B)}\int_{B}|I_{\alpha,2}((b_{1}-m_{\tilde{B}}b_{1})f_{1}^{2}(z),(b_{2}-m_{\tilde{B}}b_{2})f_{2}^{1})(z)|d\mu(z)\biggr)\\
&&\  +
C\biggl(\frac{1}{\mu(6B)}\int_{B}|I_{\alpha,2}((b_{1}-m_{\tilde{B}}b_{1})f_{1}^{2}(z),(b_{2}-m_{\tilde{B}}b_{2})f_{2}^{2})(z)-h_{B}|d\mu(z)\biggr)\\
&&=:E_{41}+E_{42}+E_{43}+E_{44}.
\end{eqnarray*}

For $1<p_i<\infty,$ $ i=1,2$, set $s_1=\sqrt{p_1}$, $s_2=\sqrt{p_2}$, $\dfrac{1}{v}=\dfrac{1}{s_1}+\dfrac{1}{s_2}-\alpha$, $\dfrac{1}{s_1}=\dfrac{1}{p_1}+\dfrac{1}{v_1}$
and $\dfrac{1}{s_2}=\dfrac{1}{p_2}+\dfrac{1}{v_2}$. It follows from H\"older's inequality and Theorem \ref{thm-main1.1} that
\begin{eqnarray*}
&&E_{41}\\
&&\leq
C\frac{\mu(B)^{1-1/v}}{\mu(6B)}||I_{\alpha,2}((b_{1}-m_{\tilde{B}}b_{1})f_{1}^{1},(b_{2}-m_{\tilde{B}}b_{2})f_{2}^{1})||_{L^{v}(\mu)}\\
&&\leq
C\frac{1}{\mu(6B)^{1/v}}||(b_{1}-m_{\tilde{B}}b_{1})f_{1}^{1}||
_{L^{s_1}(\mu)}||(b_{2}-m_{\tilde{B}}b_{2})f_{2}^{1}||_{L^{s_2}(\mu)}\\
&&\leq
C\frac{1}{\mu(6B)^{1/v}}(\int_{\frac{6}{5}B}|(b_{1}-m_{\tilde{B}}b_{1}|^{v_{1}}d\mu(z))^{1/v_{1}}(\int_{\frac{6}{5}B}|f_{1}(z)|^{p_{1}}d\mu(z))^{1/p_{1}}\\
&&\ \ \times(\int_{\frac{6}{5}B}|(b_{2}-m_{\tilde{Q}}b_{2})|^{v_{2}}d\mu(z))^{1/v_{2}}(\int_{\frac{6}{5}B}|f_{2}(z)|^{p_{2}}d\mu(z))^{1/p_{2}}\\
&&\leq
C\prod_{i=1}^{2}\bigg(\frac{\int_{\frac{6}{5}B}|b_{i}-m_{\tilde{B}}b_{i}|^{v_{i}}d\mu(z)}{\mu(6B)}\bigg)^{1/v_{i}}\bigg(\frac{\int_{\frac{6}{5}B}|f_{i}(z)|^{p_{i}}d\mu(z)}{\mu(6B)^{1-\alpha p_{i}/2}}\bigg)^{1/p_{i}}\\
&&\leq
C||b_{1}||_{\ast}||b_{2}||_{\ast}M^{(\alpha/2)}_{p_{1},(5)}f_{1}(x)M^{(\alpha/2)}_{p_{2},(5)}f_{2}(x).
\end{eqnarray*}
For $E_{42}$, using (i) of Definition 1.5, Lemma 2.2, Lemma
2.3, H\"{o}lder's inequality and the condition of $\lambda(x, ar) \geq a^{m}\lambda(x, r)$, we
have
\begin{eqnarray*}
&&E_{42}\\
&&\leq
C\frac{1}{\mu(6B)}\int_{B}\int_{X}\int_{X}
\frac{|b_{1}(y_{1})-m_{\tilde{B}}b_{1}||f_{1}^{1}(y_{1})|}
{[\lambda(z,d(z,y_{1}))+\lambda(z,d(z,y_{2}))]^{2-\alpha}}\\
&&\ \ \ \ \times|b_{2}(y_{2})-m_{\tilde{B}}b_{2}||f_{2}^{2}(y_{2})|d\mu(y_{1})d\mu(y_{2})d\mu(z)\\
&&\leq C\frac{1}{\mu(6B)}\int_{B}\int_{\frac{6}{5}B}
|b_{1}(y_{1})-m_{\tilde{B}}b_{1}||f_{1}(y_{1})|d\mu(y_{1})\\
&&\ \ \ \times\int_{X\backslash\frac{6}{5}B}
\frac{|b_{2}(y_{2})-m_{\tilde{B}}b_{2}||f_{2}(y_{2})|d\mu(y_{2})}{[\lambda(z,d(z,y_{2}))]^{2-\alpha}}d\mu(z)\\
&&\leq
C(\frac{1}{\mu(6B)}\int_{\frac{6}{5}B}|b_{1}(y_{1})-m_{\tilde{B}}b_{1}|^{p'_{1}}d\mu(y_{1}))^{1/p'_{1}}\\
&&\ \ \ \times(\frac{1}{\mu(6B)^{1-\alpha p_{1}/2}}\int_{\frac{6}{5}B}|f_{1}(y_{1})|^{p_{1}}d\mu(y_{1}))^{1/p_{1}}\\
&&\ \ \ \times
\mu(6B)^{-\alpha/2}\mu(B)\sum_{k=1}^{\infty}\int_{6^{k}\frac{6}{5}B\backslash6^{k-1}\frac{6}{5}B}\frac{|b_{2}(y_{2})-m_{\tilde{B}}
b_{2}||f_{2}(y_{2})|}{[\lambda(x,6^{k-1}\frac{6}{5}r_{B})]^{2-\alpha}}d\mu(y_{2})\\
&&\leq
C||b_{1}||_{\ast}M^{(\alpha/2)}_{p_{1},(5)}f_{1}(x)\sum_{k=1}^{\infty}6^{-km(1-\alpha/2)}
      [\frac{\mu(B)}{\mu(\frac{6}{5}B)}]^{1-\alpha/2}[\frac{\mu(\frac{6}{5}B)}{\lambda(x,\frac{6}{5}r_{B})}]^{1-\alpha/2}\\
&&\ \ \ \times\frac{1}{[\lambda(x,5\times6^{k}\frac{6}{5}r_{B})]^{1-\alpha/2}}
\int_{6^{k}\frac{6}{5}B}|b_{2}(y_{2})-m_{\tilde{B}}
b_{2}||f_{2}(y_{2})|d\mu(y_{2})\\
&&\leq
C||b_{1}||_{\ast}M^{(\alpha/2)}_{p_{1},(5)}f_{1}(x)\sum_{k=1}^{\infty}6^{-km(1-\alpha/2)}\frac{1}{[\mu(5\times6^{k}\frac{6}{5}B)]^{1-\alpha/2}}\\
&&\ \ \ \times\int_{6^{k}\frac{6}{5}B}
|b_{2}(y_{2})-m_{\widetilde{6^{k}\frac{6}{5}B}}(b_{2})+m_{\widetilde{6^{k}\frac{6}{5}B}}(b_{2})-m_{\tilde{B}}b_{2}||f_{2}(y_{2})|d\mu(y_{2})\\
&&\leq
C||b_{1}||_{\ast}M^{(\alpha/2)}_{p_{1},(5)}f_{1}(x)\sum_{k=1}^{\infty}6^{-km(1-\alpha/2)}\\
&&\ \ \ \times\biggl[\biggl(\frac{1}{\mu(5\times6^{k}\frac{6}{5}B)}\int_{6^{k}\frac{6}{5}B}|b_{2}(y_{2})-m_{\widetilde{6^{k}\frac{6}{5}B}}(b_{2})|^{p'_{2}}d\mu(y_{2})\biggr)^{1/p'_{2}}\\
&&\ \ \ \times
(\frac{1}{\mu(5\times6^{k}\frac{6}{5}B)^{1-\alpha p_{2}/2}}\int_{6^{k}\frac{6}{5}B}|f_{2}(y_{2})|^{p_{2}}d\mu(y_{2}))^{1/p_{2}}\\
&&\ \ \ +C
k||b_{2}||_{\ast}l(\frac{1}{\mu(5\times6^{k}\frac{6}{5}B)^{1-\alpha p_{2}/2}}\int_{6^{k}\frac{6}{5}B}|f_{2}(y_{2})|^{p_{2}}d\mu(y_{2}))^{1/p_{2}}\\
&&\ \ \ \times
\biggl(\frac{1}{\mu(5\times6^{k}\frac{6}{5}B)}\int_{6^{k}\frac{6}{5}B}d\mu(y_{2})\biggr)^{1/p'_{2}}\\
&&\leq
C||b_{1}||_{\ast}||b_{2}||_{\ast}M^{(\alpha/2)}_{p_{1},(5)}f_{1}(x)M^{(\alpha/2)}_{p_{2},(5)}f_{2}(x).\\
\end{eqnarray*}
Similarly, we get
\begin{equation*}
E_{43}\leq
C||b_{1}||_{\ast}||b_{2}||_{\ast}M^{(\alpha/2)}_{p_{1},(5)}f_{1}(x)M^{(\alpha/2)}_{p_{2},(5)}f_{2}(x).
\end{equation*}

For $E_{44}$, by (ii) of Definition 1.5, Lemma 2.2, Lemma 2.3,
H\"{o}lder's inequality and the properties of $\lambda$, we obtain
\begin{eqnarray*}
&&\big|I_{\alpha,2}((b_{1}-m_{\tilde{B}}b_{1})f_{2}^{2},(b_{2}-m_{\tilde{B}}b_{2})f_{2}^{2})(z)\\
&&\ \ \ \ -I_{\alpha,2}((b_{1}-m_{\tilde{B}}b_{1})f_{2}^{2},(b_{2}-m_{\tilde{B}}b_{2})f_{2}^{2})(z_{0})\big|\\
&&\leq C\int_{X\backslash
\frac{6}{5}B}\int_{X\backslash\frac{6}{5}B}|K(z,y_{1},y_{2})-K(z_{0},y_{1},y_{2})|\\
&&\ \ \ \prod_{i=1}^{2}|(b_{i}(y_{i})-m_{\tilde{B}}b_{i})f_{i}(y_{i})|d\mu(y_{i})\\
&&\leq C\int_{X\backslash
\frac{6}{5}B}\int_{X\backslash\frac{6}{5}B}\frac{d(z,z_{0})^{\delta}
\prod_{i=1}^{2}|(b_{i}(y_{i})-m_{\tilde{B}}b_{i})f_{i}(y_{i})|d\mu(y_{i}))}
{(d(z,y_{1})+d(z,y_{2}))^{\delta}[\sum_{j=1}^{2}\lambda(x,d(x,y_{j}))]^{2-\alpha}}\\
&&\leq C\prod_{i=1}^{2}\int_{X\backslash
\frac{6}{5}B}\frac{d(z,z_{0})^{\delta_{i}}
|b_{i}(y_{i})-m_{\tilde{B}}b_{i}||f_{i}(y_{i})|d\mu(y_{i})}{d(z,y_{i})^{\delta_{i}}[\lambda(z,d(z,y_{i}))]^{1-\alpha/2}}\\
&&\leq C\prod_{i=1}^{2}\sum_{k=1}^{\infty}\int_{6^{k}\frac{6}{5}B\backslash6^{k-1}\frac{6}{5}B}6^{-k\delta_{i}}[\frac{\mu(5\times6^{k}\frac{6}{5}B)}
{\lambda(z,5\times6^{k}\frac{6}{5}r_{B})}]^{1-\alpha/2}\frac{1}{[\mu(5\times6^{k}\frac{6}{5}B)]^{1-\alpha/2}}\\
&&\ \ \ \times\big|b_{i}(y_{i})-m_{\tilde{B}}b_{i}\big|\big|f_{i}\big|d\mu(y_{i})\\
&&\leq C\prod_{i=1}^{2}\sum_{k=1}^{\infty}6^{-k\delta_{i}}\bigg(\frac{1}{[\mu(5\times6^{k}\frac{6}{5}B)]^{1-\alpha p_i/2}}
\int_{6^{k}\frac{6}{5}B}|b_{i}(y_{i})-m_{\tilde{B}}b_{i}|^{p'_{i}}d\mu(y_{i})\bigg)^{1/p'_{i}}\\
&&\ \ \ \times\bigg(\frac{1}{\mu(5\times6^{k}\frac{6}{5}B)}
\int_{6^{k}\frac{6}{5}B}|f_{i}|^{p_{i}}\bigg)^{1/p_{i}}\\
&&\leq C\prod_{i=1}^{2}\sum_{k=1}^{\infty}6^{-k\delta_{i}}M_{p_{i},(6)}f_{i}(x)\bigg(\frac{1}{[\mu(5\times6^{k}\frac{6}{5}B)]^{1-\alpha p_i/2}}
\int_{6^{k}\frac{6}{5}B}|b_{i}(y_{i})-m_{\widetilde{6^{k}\frac{6}{5}B}}\\
&&\ \ \ +m_{\widetilde{6^{k}\frac{6}{5}B}}-m_{\tilde{B}}b_{i}|^{p'_{i}}d\mu(y_{i})\bigg)^{1/p'_{i}}\\
&&\leq
C\prod_{i=1}^{2}\sum_{k=1}^{\infty}6^{-k\delta_{i}}k||b_{i}||_{\ast}M_{p_{i},(6)}f_{i}(x)\\
&&\leq C||b_{1}||_{\ast}||b_{2}||_{\ast}M^{(\alpha/2)}_{p_{1},(6)}f_{1}(x)M^{(\alpha/2)}_{p_{2},(6)}f_{2}(x).
\end{eqnarray*}
where $\delta_{1},\delta_{2}>0$ and $\delta_{1}+\delta_{2}=\delta$.

Taking the mean over $z_{0}\in B$, it deduces
\begin{equation}\label{equ45}
E_{44}\leq
C||b_{1}||_{\ast}||b_{2}||_{\ast}M^{(\alpha/2)}_{p_{1},(6)}f_{1}(x)M^{(\alpha/2)}_{p_{2},(6)}f_{2}(x).
\end{equation}
So (\ref{equ33}) can be obtain from (\ref{equ36}) to (\ref{equ45}).

Next we prove (\ref{equ34}). Consider two balls $B\subset Q$ with $x\in B$,
where $B$ is an arbitrary ball and $Q$ is a doubling ball. Let
$N=N_{B,Q}+1$, then we yield

\begin{eqnarray}\label{equ46}
&&\biggl||m_{B}[I_{\alpha,2}((b_{1}-m_{\tilde{B}}b_{1})f_{1}^{2},(b_{2}-m_{\tilde{B}}b_{2})f_{2}^{2})]|\nonumber\\
&&\ \ -|m_{Q}[I_{\alpha,2}((b_{1}-m_{Q}b_{1})f_{1}^{2},(b_{2}-m_{Q}b_{2})f_{2}^{2})]|\biggr|\nonumber\\
&&\leq |m_{B}[I_{\alpha,2}((b_{1}-m_{\tilde{B}}b_{1})f_{1}\chi_{X\backslash6^{N}B},(b_{2}-m_{\tilde{B}}b_{2})f_{2}\chi_{X\backslash6^{N}B})]\nonumber\\
&&\ \ -m_{Q}[I_{\alpha,2}((b_{1}-m_{\tilde{B}}b_{1})f_{1}\chi_{X\backslash6^{N}B},(b_{2}-m_{\tilde{B}}b_{2})f_{2}\chi_{X\backslash6^{N}B})]|\nonumber\\
&&\ \
+|m_{Q}[I_{\alpha,2}((b_{1}-m_{Q}b_{1})f_{1}\chi_{X\backslash6^{N}B},(b_{2}-m_{Q}b_{2})f_{2}\chi_{X\backslash6^{N}B})]\nonumber\\
&&\ \
-m_{Q}[I_{\alpha,2}((b_{1}-m_{\tilde{B}}b_{1})f_{1}\chi_{X\backslash6^{N}B},(b_{2}-m_{\tilde{B}}b_{2})f_{2}\chi_{X\backslash6^{N}B})]|\nonumber\\
&&\ \
+|m_{B}[I_{\alpha,2}((b_{1}-m_{\tilde{B}}b_{1})f_{1}\chi_{6^{N}B\backslash\frac{6}{5}B},(b_{2}-m_{\tilde{B}}b_{2})f_{2}\chi_{X\backslash\frac{6}{5}B})]\nonumber\\
&&\ \
+|m_{B}[I_{\alpha,2}((b_{1}-m_{\tilde{B}}b_{1})f_{1}\chi_{X\backslash\frac{6}{5}B},(b_{2}-m_{\tilde{B}}b_{2})f_{2}\chi_{6^{N}B\backslash\frac{6}{5}B})]\nonumber\\
&&\ \
+|m_{Q}[I_{\alpha,2}((b_{1}-m_{Q}b_{1})f_{1}\chi_{6^{N}B\backslash\frac{6}{5}Q},(b_{2}-m_{Q}b_{2})f_{2}\chi_{X\backslash6^{N}B})]\nonumber\\
&&\ \
+|m_{Q}[I_{\alpha,2}((b_{1}-m_{Q}b_{1})f_{1}\chi_{X\backslash\frac{6}{5}Q},(b_{2}-m_{Q}b_{2})f_{2}\chi_{6^{N}B\backslash\frac{6}{5}Q})]\nonumber\\
&&=:F_{1}+F_{2}+F_{3}+F_{4}+F_{5}+F_{6}.
\end{eqnarray}

Using the method to estimate $E_{44}$, we get
\begin{equation*}
F_{1}\leq
C||b_{1}||_{\ast}||b_{2}||_{\ast}M^{(\alpha/2)}_{p_{1},(6)}f_{1}(x)M^{(\alpha/2)}_{p_{2},(6)}f_{2}(x).
\end{equation*}

Let us estimate $F_{2}$. At first, we calculate

\begin{eqnarray*}
&&I_{\alpha,2}((b_{1}-m_{Q}b_{1})f_{1}\chi_{X\backslash6^{N}B},(b_{2}-m_{Q}b_{2})f_{2}\chi_{X\backslash6^{N}B})(z)\\
&&\ \ -I_{\alpha,2}((b_{1}-m_{\tilde{B}}b_{1})f_{1}\chi_{X\backslash6^{N}B},(b_{2}-m_{\tilde{B}}b_{2})f_{2}\chi_{X\backslash6^{N}B})(z)\\
&&=(m_{Q}b_{2}-m_{\tilde{B}}b_{2})I_{\alpha,2}((b_{1}-m_{Q}b_{1})f_{1}\chi_{X\backslash6^{N}B},f_{2}\chi_{X\backslash6^{N}B})(z)\\
&&\ \ +(m_{Q}b_{1}-m_{\tilde{B}}b_{1})I_{\alpha,2}(f_{1}\chi_{X\backslash6^{N}B},(b_{2}-m_{Q}b_{2})f_{2}\chi_{X\backslash6^{N}B})(z)\\
&&\ \ +(m_{Q}b_{1}-m_{\tilde{B}}b_{1})(m_{Q}b_{2}-m_{\tilde{B}}b_{2})I_{\alpha,2}(f_{1}\chi_{X\backslash6^{N}B},f_{2}\chi_{X\backslash6^{N}B})(z).
\end{eqnarray*}

Hence
\begin{eqnarray*}
&&F_{2}\\
&&\leq|(m_{Q}b_{2}-m_{\tilde{B}}b_{2})\frac{1}{\mu(Q)}\int_{Q}I_{\alpha,2}((b_{1}-m_{Q}b_{1})f_{1}
\chi_{X\backslash6^{N}B},f_{2}\chi_{X\backslash6^{N}B})(z)d\mu(z)|\\
&&+|(m_{Q}b_{1}-m_{\tilde{B}}b_{1})\frac{1}{\mu(Q)}\int_{Q}I_{\alpha,2}((f_{1}\chi_{X\backslash6^{N}B},(b_{2}-m_{Q}b_{2})f_{2}\chi_{X\backslash6^{N}B})(z)d\mu(z)|\\
&&+|(m_{Q}b_{1}-m_{\tilde{B}}b_{1})(m_{Q}b_{2}-m_{\tilde{B}}b_{2})\frac{1}{\mu(Q)}\int_{Q}
I_{\alpha,2}(f_{1}\chi_{X\backslash6^{N}B},f_{2}\chi_{X\backslash6^{N}B})(z)|\\
&&=:F_{21}+F_{22}+F_{23}.
\end{eqnarray*}
To estimate $F_{21}$, we write
\begin{eqnarray*}
&&I_{\alpha,2}((b_{1}-m_{Q}b_{1})f_{1}\chi_{X\backslash6^{N}Q},f_{2}\chi_{X\backslash6^{N}Q})(z)\\
&&=I_{\alpha,2}((b_{1}-m_{Q}b_{1})f_{1},f_{2})(z)-T((b_{1}-m_{Q}b_{1})f_{1}\chi_{6^{N}B}\chi_{\frac{6}{5}Q},f_{2}\chi_{\frac{6}{5}Q})(z))\\
&&\ -I_{\alpha,2}((b_{1}-m_{Q}b_{1})f_{1}\chi_{\frac{6}{5}Q},f_{2}\chi_{6^{N}B}\chi_{\frac{6}{5}Q})(z))\\
&&\ +I_{\alpha,2}((b_{1}-m_{Q}b_{1})f_{1}\chi_{6^{N}B}\chi_{\frac{6}{5}Q},f_{2}\chi_{6^{N}B}\chi_{\frac{6}{5}Q})(z)\\
&&\ -I_{\alpha,2}((b_{1}-m_{Q}b_{1})f_{1}\chi_{X\backslash\frac{6}{5}Q},f_{2}\chi_{6^{N}B})(z)\\
&&\ -I_{\alpha,2}((b_{1}-m_{Q}b_{1})f_{1}\chi_{6^{N}B},f_{2}\chi_{X\backslash\frac{6}{5}Q})(z)\\
&&\ +I_{\alpha,2}((b_{1}-m_{Q}b_{1})f_{1}\chi_{6^{N}B\backslash\frac{6}{5}Q},f_{2}\chi_{6^{N}B\backslash\frac{6}{5}Q})(z)\\
 &&=:H_{1}(z)+H_{2}(z)+H_{3}(z)+H_{4}(z)+H_{5}(z)+H_{6}(z)+H_{7}(z).
\end{eqnarray*}
Let us first estimate $H_{1}(z)$. It is easy to see that
$$\frac{1}{\mu(Q)}\int_{Q}|I_{\alpha,2}(b_{1}-b_{1}(z)f_{1},f_{2})(z)|d\mu(z)\leq CM_{r,(6)}([b_{1},I_{\alpha,2}]f_{1},f_{2})(x).$$
By H\"{o}lder's inequality, we have
$$\frac{1}{\mu(Q)}\int_{Q}|(b_{1}(z)-m_{Q}(b_{1}))I_{\alpha,2}(f_{1},f_{2})(z)|d\mu(z)\leq C||b_{1}||_{\ast}M_{r,(6)}(I_{\alpha,2}(f_{1},f_{2}))(x).$$
Then we obtain
\begin{eqnarray*}
&&|m_{Q}(H_{1})|\\
&\leq&
|m_{Q}(I_{\alpha,2}(b_{1}-b_{1}(z)f_{1},f_{2}))|+|m_{Q}((b_{1}(z)-m_{Q}(b_{1}))I_{\alpha,2}(f_{1},f_{2}))|\\
&\leq&
CM_{r,(6)}([b_{1},I_{\alpha,2}]f_{1},f_{2})(x)+||b_{1}||_{\ast}M_{r,(6)}(I_{\alpha,2}(f_{1},f_{2}))(x).
\end{eqnarray*}
For $H_{2}(z)$,  $s_1=\sqrt{p_1}$, $s_2=p_2$, $\dfrac{1}{v}=\dfrac{1}{s_1}+\dfrac{1}{s_2}-\alpha$ and $\dfrac{1}{s_1}=\dfrac{1}{p_1}+\dfrac{1}{v_1}$. Using the fact
that $Q$ is a doubling balls, Lemma 3.1 and H\"{o}lder's
inequality, we yield

\begin{eqnarray*}
&&|m_{Q}(H_{2})|\\
&\leq&
C\frac{\mu(Q)^{1-1/v}}{\mu(6Q)}||I_{\alpha,2}((b_{1}-m_{Q}b_{1})f_{1}\chi_{6^{N}B}\chi_{6/5Q},f_{2}\chi_{6/5Q})||_{L^{v}(\mu)}\\
&\leq&
C\mu(6Q)^{-1/v}||(b_{1}-m_{Q}b_{1})f_{1}\chi_{6^{N}B}\chi_{6/5Q}||
_{L^{s_1}(\mu)}||f_{2}\chi_{6/5Q}||_{L^{s_2}(\mu)}\\
&\leq&
C\frac{1}{\mu(6Q)^{1/v}}(\int_{\frac{6}{5}Q}|(b_{1}-m_{Q}b_{1}|^{v_{1}}d\mu(z))^{1/v_{1}}(\int_{\frac{6}{5}Q}|f_{1}(z)|^{p_{1}}d\mu(z))^{1/p_{1}}\\
&&\ \ \times(\int_{\frac{6}{5}Q}|f_{2}(z)|^{p_{2}}d\mu(z))^{1/p_{2}}\\
&\leq&
C(\frac{1}{\mu(6Q)}\int_{\frac{6}{5}Q}|b_{1}-m_{Q}b_{1}|^{v_{1}}d\mu(z))^{1/v_{1}}\\
&&\ \ \times\prod_{i=1}^{2}(\frac{1}{\mu(6Q)^{1-\alpha p_{i}/2}}\int_{\frac{6}{5}Q}|f_{i}(z)|^{p_{i}}d\mu(z))^{1/p_{i}}\\
&\leq&
C||b_{1}||_{\ast}M^{(\alpha/2)}_{p_{1},(5)}f_{1}(x)M^{(\alpha/2)}_{p_{2},(5)}f_{2}(x).
\end{eqnarray*}

We can also obtain
\begin{equation*}
|m_{Q}(H_{3})|+|m_{Q}(H_{4})|\leq
C||b_{1}||_{\ast}M^{(\alpha/2)}_{p_{1},(5)}f_{1}(x)M^{(\alpha/2)}_{p_{2},(5)}f_{2}(x).
\end{equation*}
For $H_{5}$, since $z\in Q$, by (i) of Definition 1.5, Lemma 2.2,
Lemma 2.3, H\"{o}lder's inequality and the properties of $\lambda$
and $Q$ is a doubling ball, we deduce
\begin{eqnarray*}
&&|H_{5}(z)|\\
&&\leq C
\int_{6^{N}B}\int_{X\backslash\frac{6}{5}Q}\frac{|b_{1}(y_{1})-m_{Q}b_{1}||f_{1}(y_{1})||f_{2}(y_{2})|d\mu(y_{1})d\mu(y_{2})}
{[\sum_{j=1}^{2}\lambda(x,d(x,y_{j}))]^{2-\alpha}}\\
&&\leq
C\int_{6^{N}B}|f_{2}(y_{2})|d\mu(y_{2})\sum_{k=1}^{\infty}\int_{6^{k}\frac{6}{5}Q}
\frac{|b_{1}(y_{1})-m_{Q}b_{1}||f_{1}(y_{1})|}{(\lambda(z,5\times6^{k}\frac{6}{5}r_{Q}))^{2-\alpha}}d\mu(y_{1})\\
&&\leq
C\int_{6^{N}B}|f_{2}(y_{2})|d\mu(y_{2})\sum_{k=1}^{\infty}6^{-km(1-\alpha/2)}\\
&&\ \ \ \times \int_{6^{k}\frac{6}{5}Q}
\frac{1}{[\lambda(z,5\times\frac{6}{5}r_{Q})]^{1-\alpha/2}}\frac{|b_{1}(y_{1})-m_{Q}b_{1}||f_{1}(y_{1})|d\mu(y_{1})}{[\lambda(z,5\times6^{k}\frac{6}{5}r_{Q})]^{1-\alpha/2}}\\
&&\leq
C\frac{1}{[\lambda(z,6r_{Q})]^{1-\alpha/2}}\int_{6^{N}B}|f_{2}(y_{2})|d\mu(y_{2})\sum_{k=1}^{\infty}6^{-km(1-\alpha/2)}\\
&&\ \ \times\frac{1}{[\lambda(z,5\times6^{k}\frac{6}{5}r_{Q})]^{1-\alpha/2}}\times
\biggl[\int_{6^{k}\frac{6}{5}Q}|b_{1}(y_{1})-m_{6^{k}\frac{6}{5}Q}(b_{1})||f_{1}(y_{1})|d\mu(y_{1})\\
&&\ \ \ +\int_{6^{k}\frac{6}{5}Q}|m_{6^{k}\frac{6}{5}Q}(b_{1})-m_{Q}b_{1}||f_{1}(y_{1})|d\mu(y_{1})\biggr]\\
&&\leq
C\frac{1}{[\lambda(z,6r_{Q})]^{1-\alpha/2}}\int_{6^{N}B}|f_{2}(y_{2})|d\mu(y_{2})\sum_{k=1}^{\infty}6^{-km(1-\alpha/2)}\\
&&\ \ \
\times\biggl[\biggl(\frac{1}{\lambda(z,5\times6^{k}\frac{6}{5}r_{Q})}\int_{6^{k}\frac{6}{5}Q}
|b_{1}(y_{1})-m_{6^{k}\frac{6}{5}Q}(b_{1})|^{p'_{1}}d\mu(y_{1})\biggr)^{1/p'_{1}}\\
&&\ \ \ \ \times\biggl(\frac{1}{[\lambda(z,6^{k+1}\frac{6}{5}r_{Q})]^{1-\alpha p_{1}/2}}\int_{6^{k}\frac{6}{5}Q}|f_{1}(y_{1})|^{p_{1}}d\mu(y_{1})\biggr)^{1/p_{1}}\\
&&\ \ \ \ +k||b_{1}||_{\ast}\frac{1}{[\lambda(z,5\times6^{k}\frac{6}{5}r_{Q})]^{1-\alpha/2}}\int_{6^{k}\frac{6}{5}Q}|f_{1}(y_{1})|d\mu(y_{1})\biggr]\\
&&\leq
C||b_{1}||_{\ast}M^{(\alpha/2)}_{p_{1},(5)}f_{1}(x)\sum_{k=1}^{N}\frac{1}{[\lambda(z,6r_{Q})]^{1-\alpha/2}}\int_{6^{k}B}|f_{2}(y_{2})|d\mu(y_{2})\\
&&\leq
C||b_{1}||_{\ast}M^{(\alpha/2)}_{p_{1},(5)}f_{1}(x)\sum_{k=1}^{N}[\frac{\mu(5\times6^{k}B)}{\lambda(z,5\times6^{k}r_{B})}]^{1-\alpha/2}[\frac{\lambda(z,5\times6^{k}r_{B})}{\lambda(z,6r_{Q})}]^{1-\alpha/2}\\
&&\ \ \ \ \frac{1}{[\mu(5\times6^{k}B)]^{1-\alpha/2}}\int_{6^{k}B}|f_{2}(y_{2})|d\mu(y_{2})\\
&&\leq
CK^{(\alpha/2)}_{B,Q}||b_{1}||_{\ast}M^{(\alpha/2)}_{p_{1},(5)}f_{1}(x)M^{(\alpha/2)}_{p_{2},(5)}f_{2}(x).
\end{eqnarray*}
Then
\begin{equation*}
|m_{Q}(H_{5})|\leq
CK^{(\alpha/2)}_{B,Q}||b_{1}||_{\ast}M^{(\alpha/2)}_{p_{1},(5)}f_{1}(x)M^{(\alpha/2)}_{p_{2},(5)}f_{2}(x).
\end{equation*}
In the similar way to estimate $m_{Q}(H_{5})$, it follows that
\begin{equation*}
|m_{Q}(H_{6})|+|m_{Q}(H_{7})|\leq
CK^{(\alpha/2)}_{B,Q}||b_{1}||_{\ast}M^{(\alpha/2)}_{p_{1},(5)}f_{1}(x)M^{(\alpha/2)}_{p_{2},(5)}f_{2}(x).
\end{equation*}
From (\ref{equ44}) in Lemma 2.2, we deduce
\begin{eqnarray*}
&&F_{21}\\
&&\leq C\biggr[||b_{1}||_{\ast}||b_{2}||_{\ast}M_{r,(6)}(I_{\alpha,2}(f_{1},f_{2}))(x)
+||b_{1}||_{\ast}M_{r,(6)}([b_{2},I_{\alpha,2}](f_{1},f_{2}))(x)\\
&&+||b_{2}||_{\ast}M_{r,(6)}([b_{1},I_{\alpha,2}](f_{1},f_{2}))(x)
+||b_{1}||_{\ast}||b_{2}||_{\ast}M^{(\alpha/2)}_{p_{1},(5)}f_{1}(x)M^{(\alpha/2)}_{p_{2},(5)}f_{2}(x)\biggr].
\end{eqnarray*}
$F_{22}$ and $F_{23}$ also have similar estimate of $F_{21}$,
therefore,
\begin{eqnarray*}
&&F_{2}\\
&&\leq C\biggr[||b_{1}||_{\ast}||b_{2}||_{\ast}M_{r,(6)}(I_{\alpha,2}(f_{1},f_{2}))(x)
+||b_{1}||_{\ast}M_{r,(6)}([b_{2},I_{\alpha,2}](f_{1},f_{2}))(x)\\
&&+||b_{2}||_{\ast}M_{r,(6)}([b_{1},I_{\alpha,2}](f_{1},f_{2}))(x)
+||b_{1}||_{\ast}||b_{2}||_{\ast}M^{(\alpha/2)}_{p_{1},(5)}f_{1}(x)M^{(\alpha/2)}_{p_{2},(5)}f_{2}(x)\biggr].
\end{eqnarray*}
From $F_{3}$ to $F_{6}$, using the similar method to estimate
$I_{4}$, we conclude

\begin{equation}\label{equ59}
F_{3}+F_{4}+F_{5}+F_{6}\leq
C||b_{1}||_{\ast}||b_{2}||_{\ast}M^{(\alpha/2)}_{p_{1},(5)}f_{1}(x)M^{(\alpha/2)}_{p_{2},(5)}f_{2}(x).
\end{equation}
Thus (\ref{equ34}) holds by from (\ref{equ46}) to (\ref{equ59}) and hence (\ref{equ30}) is proved.  With
the same method to prove (\ref{equ30}), we can obtain that (\ref{equ31}) and (\ref{equ32}) are
also hold. Here we omit the details. Thus Lemma 2.4 is proved.
\end{proof}

\begin{proof}[Proof of Theorem \ref{thm-main1.2}]
 Let $1<p_{1},\ p_{2},\ q<\infty$,
$\dfrac{1}{q}=\dfrac{1}{p_{1}}+\dfrac{1}{p_{2}}$, $1<r<q$,
 $f_{1}\in
L^{p_{1}}(\mu)$, $f_{2}\in L^{p_{2}}(\mu)$, $b_{1}\in RBMO(\mu)$ and
$b_{2}\in RBMO(\mu)$.
 By $|f(x)|\leq
Nf(x)$, Lemma 2.1-2.4, Theorem \ref{thm-main1.1}, H\"{o}rder's
inequality and the boundedness of $M^{(\alpha/2)}_{r,(\rho)}$ and $M_{r,(\rho)}$
for $0<\alpha<2$, $5 \leq\rho$ and $r<q$, we obtain
\begin{eqnarray*}
&&||[b_{1},b_{2},I_{\alpha,2}](f_{1},f_{2})||_{L^{q}(\mu)}\\
&\leq& ||N([b_{1},b_{2},I_{\alpha,2}](f_{1},f_{2}))||_{L^{q}(\mu)}\\
&\leq& C||M^{\sharp,(\alpha/2)}([b_{1},b_{2},I_{\alpha,2}](f_{1},f_{2}))||_{L^{q}(\mu)}\\
&\leq& C||b_{1}||_{\ast}||b_{2}||_{\ast}||M_{r,(6)}(I_{\alpha,2}(f_{1},f_{2}))||_{L^{q}(\mu)}\\
&&+C||b_{1}||_{\ast}||M_{r,(6)}([b_{2},I_{\alpha,2}](f_{1},f_{2}))||_{L^{q}(\mu)}\\
&&+C||b_{2}||_{\ast}||M_{r,(6)}([b_{1},I_{\alpha,2}](f_{1},f_{2}))||_{L^{q}(\mu)}\\
&&+C||b_{1}||_{\ast}||b_{2}||_{\ast}||M^{(\alpha/2)}_{p_{1},(5)}f_{1}(x)M^{(\alpha/2)}_{p_{2},(5)}f_{2}(x)||_{L^{q}(\mu)}\\
&\leq& C||b_{1}||_{\ast}||b_{2}||_{\ast}||f_{1}(x)||_{L^{p_{1}}(\mu)}||f_{2}(x)||_{L^{p_{2}}(\mu)}\\
&& +C||b_{1}||_{\ast}||([b_{2},I_{\alpha,2}](f_{1},f_{2}))||_{L^{q}(\mu)}\\
&&+C||b_{2}||_{\ast}||([b_{1},I_{\alpha,2}](f_{1},f_{2}))||_{L^{q}(\mu)}\\
&\leq& C||b_{1}||_{\ast}||b_{2}||_{\ast}||f_{1}(x)||_{L^{p_{1}}(\mu)}||f_{2}(x)||_{L^{p_{2}}(\mu)}\\
&& +C||b_{1}||_{\ast}||M^{\sharp,(\alpha/2)}([b_{2},I_{\alpha,2}](f_{1},f_{2}))||_{L^{q}(\mu)}\\
&&+C||b_{2}||_{\ast}||M^{\sharp,(\alpha/2)}([b_{1},I_{\alpha,2}](f_{1},f_{2}))||_{L^{q}(\mu)}\\
&\leq& ||b_{1}||_{\ast}||b_{2}||_{\ast}||f_{1}(x)||_{L^{p_{1}}(\mu)}||f_{2}(x)||_{L^{p_{2}}(\mu)}\\
&&+C||b_{1}||_{\ast}||M_{r,(6)}(I_{\alpha,2}(f_{1},f_{2}))(x)||_{L^{q}(\mu)}\\
&&+C||b_{1}||_{\ast}||M^{(\alpha/2)}_{p_{1},(5)}f_{1}(x)M^{(\alpha/2)}_{p_{2},(5)}f_{2}(x)||_{L^{q}(\mu)}\\
&&+C||b_{2}||_{\ast}||M_{r,(6)}(I_{\alpha,2}(f_{1},f_{2}))(x)||_{L^{q}(\mu)}\\
&&+C||b_{2}||_{\ast}||M^{(\alpha/2)}_{p_{1},(5)}f_{1}(x)M^{(\alpha/2)}_{p_{2},(5)}f_{2}(x)||_{L^{q}(\mu)}\\
&\leq& C||b_{1}||_{\ast}||b_{2}||_{\ast}||f_{1}(x)||_{L^{p_{1}}(\mu)}||f_{2}(x)||_{L^{p_{2}}(\mu)}.
\end{eqnarray*}
Thus the proof of Theorem \ref{thm-main1.2} is finished.
\end{proof}

\section{Applications}
In this section, we apply Theorems \ref{thm-main1.1}
and Theorem \ref{thm-main1.2} to study a specific fractional
integral operator.
\begin{lemma}\cite{FYY2}\label{l3.1}
 Assume that   $diam(X)=\infty$.
 Let $\alpha\in(0,1)$, $p\in(1,1/\alpha)$ and $1/q=1/p-\alpha$.
If $\lambda$ satisfies the
{$\epsilon$-weak reverse doubling condition}
for some $\epsilon\in(0,\min\{\alpha,1-\alpha,1/q\})$, then
\begin{equation*}
T_{\alpha} f(x):=\int_X \frac{f(y)}{[\lambda(y,d(x,y))]^{1-\alpha}}\,d\mu(y).
\end{equation*}
is bounded from $\L^{p}(\mu)$ into $L^{q}(\mu)$.
\end{lemma}

\begin{theorem}\label{}
Under the same assumption as that of Lemma \ref{l3.1},
 the conclusions of Theorem \ref{thm-main1.1} and
Theorems \ref{thm-main1.2} hold true, if $I_\alpha$ therein is replaced by $T_\alpha$.
\end{theorem}

{\bf Acknowledgements.}\quad Huajun Gong is partial supported by Posted-doctor Foundation of China(No. 2015M580728) and Shenzhen University (No. 2014-62). Rulong Xie is partial supported by NSF of Anhui Province (No. 1608085QA12). Chen Xu is partial supported by No.2013B040403005; No.GCZX-A1409 and NSFC £¨No.61472257£©.

\bibliographystyle{amsplain}

\begin{thebibliography}{99}

\bibitem{BD}Bui T, Duong X. Hardy spaces, regularized BMO spaces and the
boundedness of Calder\'{o}n-Zygmund operators on non-homogeneous
spaces. J Geom Anal, 2013, 23: 895-932.

\bibitem{CS}Chen W, Sawyer E.
A note on commutators of fractional integrals with RBMO$(\mu)$
functions.  Illinois Journal of Mathematics, 2002, 46: 1287-1298.



\bibitem{CM}Coifman R, Meyer Y. On commutators of singular integrals and
bilinear singular integrals. Trans Amer Math Soc, 1975, 212: 315-331.

\bibitem{FLYY}Fu X, Lin H, Yang Da, Yang Do. Hardy spaces $\mathcal H^p$ over non-homogeneous metric measure spaces and their applications. Sci. China Math, 2015, 58: 309-388.

\bibitem{FYY1}Fu X, Yang D, Yuan W. Boundedness of multilinear
commutators of Calder\'{o}n-Zygmund operators on orlicz spaces over
non-homogeneous spaces. Taiwanese J. of Math, 2012, 16:
2203-2238.

\bibitem{FYY2}Fu X, Yang D, Yuan W. Generalized fractional integrals and their commutators over
non-homogeneous metric measure spaces . Taiwanese J. of Math, 2014, 18: 509-557.


\bibitem{GM} Garc\'{i}a-Cuerva J, Mar\'{i}a-Martell J.
Two-weight norm inequalities for maximal operators and fractional
integrals on non-homogeneous spaces. Indiana University Mathematics
Journal, 2001, 50: 1241-1280.



\bibitem{GT1}Grafakos L, Torres R. Multilinear Calder\'{o}n-Zygmund theory.
Adv Math, 2002, 165: 124-164.

\bibitem{GT2}Grafakos L, Torres R. On multilinear singular integrals of
Calder\'{o}n-Zygmund type. Publ Mat, 2002, 57-91. (extra)


\bibitem{HMY1}Hu G, Meng Y, Yang D. Multilinear commutators of singular
integrals with non doubling measures. Integral Equations Operator
Theory, 2005, 51: 235-255.


\bibitem{HMY2} Hu G, Meng Y, Yang D. New atomic characterization of $\mathcal H^{1}$
space with nondoubling measures and its applications, Math Proc
Cambridge Philos Soc, 2005,  138: 151-171.

\bibitem{HMY3} Hu G, Meng Y, Yang D.  Weighted norm inequalities for multilinear Calder\'{o}n-Zygmund operators
on non-homogeneous metric measure spaces,  Forum Math, 2012, doi: 10. 1515/forum-2011-0042.


\bibitem{H}Hyt\"{o}nen T. A framework for non-homogeneous analysis on metric
spaces, and the RBMO space of Tolsa. Publ Mat, 2010, 54: 485-504.


\bibitem{HM}Hyt\"{o}nen T, Martikainen H. Non-homogeneous $T_b$ theorem and random dyadic cubes on metric measure spaces.
 J Geom Anal, 2012, 22: 1071-1107.


\bibitem{HLYY}Hyt\"{o}nen T, Liu  S, Yang Da, Yang Do. Boundedness of
Calder\'{o}n-Zygmund operators on non-homogeneous metric measure
spaces. Canad J Math, DOI 10.4153/CJM- 2011-065-2 (arXiv:
1011.2937)


\bibitem{HYY}Hyt\"{o}nen T, Yang Da, Yang Do. The Hardy space $\mathcal H^{1}$ on
non-homogeneous metric spaces. Mathematical Proceedings of the
Cambridge Philosophical Society, 2012, 153(01): 9-31.

\bibitem{LW} Lian J, Wu H. A class of commutators for multilinear fractional integrls in nonhomogeneous spaces. J. Inequ. Appl, 2008, Article ID 373050.



\bibitem{LY1}Lin H, Yang D. Spaces of type BLO on non-homogeneous metric measure spaces. Front Math China, 2011, 6: 271-292.

\bibitem{LY2} Lin H, Yang D.  An interpolation
theorem for sublinear operators on
non-homogeneous metric measure spaces,
Banach J. Math. Anal, 2012, 6: 168-179.

\bibitem{LY3}Lin H, Yang D. Equivalent boundedness of Marcinkiewicz integrals on
non-homogeneous metric measure spaces. Sci. China Math, 2014, 57: 123-144.

\bibitem{LYY}Liu S, Yang Da, Yang Do. Boundedness of
Calder\'{o}n-Zygmund operators on non-homogeneous metric measure
spaces: equivalent characterizations. J Math Anal Appl, 2012, 386:
258-272.

\bibitem{NTV} Nazarov F, Treil S, Volberg A. The $T_b$-theorem on
non-homogeneous spaces. Acta Math, 2003, 190: 151-239.

\bibitem{TL} Tan C,  Li J.
Littlewood-Paley theory on metric measure
spaces with non doubling measures and its applications.
Sci. China Math, 2015, 58: 983-1004.



\bibitem{T1} Tolsa X. Painlev\'{e}'s problem and the semiadditivity of analytic
capacity. Acta Math, 2003, 190: 105-149.



\bibitem{T2} Tolsa X.  BMO, $\mathcal H^{1}$ and Calder\'{o}n-Zygmund operators for non-doubling
measures. Math Ann, 2001, 319: 89-101.


\bibitem{XGZH}Xie R, Gong H, Zhou X.  Commutators of multilinear singular integral operators on non-homogeneous metric measure spaces.
Taiwanese J. Math., 2015, 19: 703-723.

\bibitem{XS}Xie R, Shu L. $\Theta$-type Calder\'{o}n-Zygmund
operators with non-doubling measures. Acta Math Appl Sinica, English
Series, 2013, 29: 263-280.

\bibitem{X1} Xu J. Boundedness of multilinear singular integrals for
non-doubling measures. J Math Anal Appl, 2007, 327: 471-480.

\bibitem{X2}Xu J. Boundedness in Lebesgue spaces for commutators of
multilinear singular integrals and RBMO functions with non-doubling
measures. Science in China. Series A, 2007, 50: 361-376.

\bibitem{YYF} Yang Da, Yang Do, Fu X.
The Hardy space $\mathcal H^{1}$  on non-homogeneous spaces
and its applications---a survey,
Eurasian Math. J., 2013, 4: 104-139.

\bibitem{YYH} Yang Da, Yang Do, Hu G. The Hardy space $\mathcal H^1$ with non-doubling measures
and their applications, Lecture Notes in Mathematics 2084,
Springer-Verlag, Berlin, 2013.


\end{thebibliography}

\end{document}